\documentclass[10pt,reqno]{amsart}
\usepackage{amssymb}
\usepackage{amsmath}
\usepackage{mathrsfs}
\usepackage{amssymb, amsmath, amsfonts}

\usepackage{mathabx}

\usepackage{color}

\usepackage[hidelinks]{hyperref}
\allowdisplaybreaks
\usepackage[numbers,sort&compress]{natbib}

\makeatletter
\def\tank#1{\protected@xdef\@thanks{\@thanks
 \protect\footnotetext[0]{#1}}}
\def\bigfoot{

 \@footnotetext}
\makeatother

\newcommand{\ea}{\end{array}}

\allowdisplaybreaks
\numberwithin{equation}{section}
\newtheorem{theorem}{Theorem}[section]
\newtheorem{lemma}{Lemma}[section]
\newtheorem{proposition}{Proposition}[section]

\newtheorem{corollary}{Corollary}[section]
\newtheorem{remark}{Remark}[section]

\newtheorem{condition}{Condition}[section]

\def\beq{\begin{equation}}
\def\nneq{\end{equation}}

\def\bthm{\begin{theorem}}
\def\nthm{\end{theorem}}

\def\blem{\begin{lemma}}
\def\nlem{\end{lemma}}
\def\bprf{\begin{proof}}
\def\nprf{\end{proof}}
\def\bprop{\begin{prop}}
\def\nprop{\end{prop}}
\def\brmk{\begin{rem}}
\def\nrmk{\end{rem}}

\def\bexa{\begin{exa}}
\def\nexa{\end{exa}}
\def\bcor{\begin{cor}}
\def\ncor{\end{cor}}

\def\e{\varepsilon}

\def\RR{\mathbb{R}}
\def\RR{\mathbb{R}}

\def\EE{\mathbb{E}}

\def\cF{\mathcal{F}}

\def\cD{\mathcal{D}}

\def\cN{\mathcal{N}}

\newcommand\HH{\mathcal H}

\def\e{{\varepsilon}}

\newcommand{\lc}{\left(}
\newcommand{\rc}{\right)}
\newcommand{\lk}{\left[}
\newcommand{\rk}{\right]}
\newcommand{\lt}{\left }
\newcommand{\rt}{\right}

\title[]{ Temporal properties of the stochastic fractional heat equation  with  rough dependence in space}

\author[B. Zhang]{Beibei Zhang}
    \address[]{Beibei Zhang, School of Mathematics and Statistics, Suzhou University of Technology, Changshu, Jiangsu, 215500,  China.}
    \email{zhangbb@whu.edu.cn}

\author[B. Qian]{Bin Qian}
   \address[]{Bin Qian, School of Mathematics and Statistics, Suzhou University of Technology, Changshu, Jiangsu, 215500,  China.}
    \email{binqiancn@126.com}

\date{}

\begin{document}
\maketitle

 \noindent {\bf Abstract:}

This paper investigates the nonlinear stochastic fractional heat equation  driven by a Gaussian noise that is white in time and fractional in space with a Hurst parameter $H \in \big(\frac{3-\alpha}{4}, \frac{1}{2}\big)$. Specifically, the driving operator is the fractional Laplacian of order $\alpha/2 \in (1/2, 1)$. We characterize the asymptotic behavior of the temporal increment $u(t+\varepsilon,x)-u(t,x)$ for fixed $t\ge 0$ and $x\in\mathbb{R}$ as $\varepsilon\downarrow 0$. Utilizing these precise asymptotic estimates, we establish Khinchin's and Chung's laws of the iterated logarithm for the temporal process $t \mapsto u(t,x)$.

 \vskip0.3cm
 \noindent{\bf Keywords:} {Stochastic heat equation; Fractional Brownian motion; Law of the iterated logarithm; Rough noise}
 \vskip0.3cm

\noindent {\bf MSC: }
 {60H15; 60G17; 60G22}
\vskip0.3cm

\section{Introduction  }
In this paper, we investigate the following nonlinear stochastic fractional heat equation (SFHE):
\begin{equation}\label{SFHE}
  \frac{\partial u(t,x)}{\partial t}
  =
  -(-\Delta)^{\frac{\alpha}{2}}u(t,x)
  +\sigma(u(t,x))\dot{W}(t,x),
  \quad t>0,\ x\in\mathbb{R},
\end{equation}
with a given initial condition $u_0$. Here,
$-(-\Delta)^{\frac{\alpha}{2}}$ denotes the fractional Laplacian of order
$\frac{\alpha}{2}\in(\frac12,1)$, and $\dot W$ is a Gaussian noise which is white in time and fractional in space with Hurst parameter $H\in\left(\frac{3-\alpha}{4},\frac12\right)$.
More precisely, the covariance structure of the underlying Gaussian field is given by
\begin{equation}\label{CovW}
\mathbb E[W(t,x)W(s,y)]
=
\frac12(s\wedge t)
\left(
|x|^{2H}+|y|^{2H}-|x-y|^{2H}
\right).
\end{equation}
The spatial covariance of the formal derivative
$\dot W(t,x)$ is described by the spectral measure $\mu(d\xi)=c_{1,1}|\xi|^{1-2H}d\xi$, where
\begin{equation}\label{e.c1}
c_{1,1}
=
\frac1{2\pi}
\Gamma(2H+1)\sin(\pi H).
\end{equation}

The analysis of stochastic partial differential equations driven by rough spatial noises has attracted considerable attention in recent years. When
$H<\frac12$, the spatial covariance becomes singular, and the construction of stochastic integrals requires refined techniques beyond the standard function-valued framework \cite{Dalang1999}.

 Nevertheless, well-posedness results for stochastic heat equations with rough spatial noises have been developed in a series of works, including affine equations
\cite{BJQ2015*,BJQ2016*}
and nonlinear equations
\cite{HHLNT2017,HHLNT2018,HW2022}.
In particular, Liu and Mao
\cite{LM2022}
established existence and uniqueness for equation \eqref{SFHE} under assumptions on the initial condition and nonlinear coefficient similar to those considered here.

After establishing well-posedness, a natural question is to understand the fine sample-path behavior of the solution. In particular, the temporal regularity and local oscillation properties of stochastic heat equations have been extensively studied. A related approximation principle was developed by Khoshnevisan et al.
\cite{KSXZ2013}
for differential equations driven by a very rough fractional Brownian motion (fBm). Their approach shows that, for fixed $x\in\mathbb R$, the temporal increment $u(t+\varepsilon,x)-u(t,x)$ can be approximated by the increment of a  fBm. More precisely, the local behavior of the solution is governed by a Gaussian process with the same covariance structure as an appropriate fBm, while the remaining term is smoother.

This approximation principle has been further extended to several related models. Das
\cite{Das2022*}
applied this idea to the Kardar--Parisi--Zhang equation, while Wang and Xiao
\cite{WX2024}
considered SFHEs with spatially correlated noises satisfying
$H>\frac12$.
More recently, Qian et al.
\cite{QianWAngXiao2026}
studied the standard stochastic heat equation driven by rough spatial noises. However, the temporal local behavior of SFHEs driven by rough spatial noises remains insufficiently understood. The main difficulty arises from the interplay between the nonlocal nature of the fractional Laplacian and the singularity of the spatial covariance structure, which makes a direct application of existing arguments difficult.

The aim of this paper is to further investigate the temporal local behavior of the nonlinear SFHE in the rough spatial regime
$\alpha\in(1,2)$ and $H\in\left(\frac{3-\alpha}{4},\frac12\right)$.
The restriction on $H$ arises naturally from the integrability properties of the fractional Sobolev seminorm associated with the noise. Under this condition, the effective temporal Hurst parameter
$\frac{\widetilde H}{2}=\frac{\alpha+2H-2}{2\alpha}$ is positive and determines the scaling behavior of the dominant temporal fluctuations.

The main results of this paper can be summarized as follows. We obtain an approximation estimate for the temporal increment of the nonlinear solution, showing that the leading-order fluctuation can be described by the corresponding linear stochastic convolution multiplied by the coefficient $\sigma(u(t,x))$. Furthermore, by combining this approximation with the Gaussian decomposition of the linear solution, we derive Khinchin's and Chung's laws of the iterated logarithm for the temporal process $t\mapsto u(t,x)$, showing that its local oscillation behavior is governed by the same scaling as a fractional Brownian motion with index $\widetilde H/2$.

The proof is based on a careful decomposition of the temporal increment into several stochastic convolution terms. We first analyze the contribution from the stochastic integral over a short time interval and then approximate the historical part by replacing the nonlinear coefficient with its value at $(t,x)$. The main challenge is to obtain sufficiently sharp estimates for the resulting approximation error under the singular spatial covariance structure. Our analysis relies on a combination of fractional Green function estimates, fractional Sobolev estimates, and scaling properties of the stable heat kernel.

For the linear equation
\begin{equation}\label{eq SFHE linear}
\frac{\partial v(t,x)}{\partial t}
=
-(-\Delta)^{\frac{\alpha}{2}}v(t,x)
+\dot W(t,x),
\qquad
v(0,\cdot)=0,
\end{equation}
the temporal process admits a decomposition into a fractional Brownian motion and a smooth Gaussian remainder. More precisely, for fixed $x\in\mathbb R$,
\begin{equation}\label{eq decom}
t\mapsto
\kappa B_t^{\widetilde H/2}-v(t,x)
\end{equation}
has a version which is infinitely differentiable on $(0,\infty)$, where
\begin{equation}\label{eq constant 1}
 \kappa:=\left(  \frac{ \Gamma(2\widetilde{H})}{ \Gamma(\widetilde{H})} \right)^{\frac12}.
 \end{equation}
Therefore, the local behavior of the nonlinear solution can be transferred from the linear Gaussian process to the corresponding fractional Brownian motion.

The rest of the paper is organized as follows. Section 2 introduces the stochastic integration framework, properties of the fractional heat kernel, and basic estimates for the mild solution. Section 3 presents the main approximation theorem and its applications to the laws of the iterated logarithm. Section 4 is devoted to the proof of the main result. Additional technical estimates are collected in the Appendix.

\section{Preliminaries}\label{lemmas}

\subsection{Covariance structure and stochastic integration}

 Recall  some notations from \cite{HHLNT2017} and \cite{HW2022}. Denote by  $ \cD=\cD(\RR) $  the space of   {real-valued infinitely differentiable} functions with compact support on $\mathbb{R}$. The Fourier transform of a function $f\in\cD$ is defined as
 $$
  \cF f(\xi):=\int_{\RR} e^{-i\xi x}f(x) dx.
$$

 Let $\mathcal H$  be  the Hilbert space obtained by completing $\mathcal D(\RR) $ with respect to the following scalar  product: for $\forall \phi,\psi\in \mathcal D(\mathbb R)$,
	  \begin{equation}\label{eq H product}
\begin{split}
     \langle\phi,\psi\rangle_{\mathcal H}
     =&\,c_{1,1}\int_{\RR} \cF \phi(\xi) \overline{\cF \psi(\xi)}  \cdot |\xi|^{1-2H}d\xi\\
     =&\, H\left(\frac12-H\right)\int_{\RR^2}[\phi(x+y)-\phi(x)]\cdot [\psi(x+y)-\psi(x)]\cdot |y|^{2H-2} dxdy,
        \end{split}
   \end{equation}
  where $H\in (\frac14 , \frac 12)$, $c_{1,1}$ is given in \eqref{e.c1}.

 Let   $(\Omega,\cF,\mathbb P)$ be a complete probability space and   { $\mathcal{D}(\mathbb{R}_{+}\times \mathbb{R})$   the space of real-valued infinitely differentiable functions with compact support on $\mathbb{R}_{+}\times \mathbb{R}$.}
The noise $\dot W$ is  a zero-mean Gaussian family $\{W(\phi), \phi\in\cD(\RR_{+}\times \RR)\}$  with the covariance structure given by
 \begin{equation}\label{CovStru}
   \EE\big[ W(\phi)W(\psi)\big]=c_{1,1}\int_{\RR_{+}\times \RR} \cF \phi(s,\xi) \overline{\cF \psi(s,\xi)}\cdot |\xi|^{1-2H}d\xi ds,
 \end{equation}
 where $H\in (\frac14 , \frac 12)$, $c_{1,1}$ is given in \eqref{e.c1}.  For any $t\ge0$, let $\cF_t$ be the filtration generated by $W$, that is
$$
 \cF_t:=\sigma\big\{W(\phi): \phi\in \cD([0,t]\times\RR)\big\},
$$
where $\cD([0,t]\times\RR)$ is  the space of real-valued infinitely differentiable functions on $[0,t]\times\RR$.

\begin{proposition}(\cite[Proposition 2.3]{HHLNT2017})\label{prop 2.3}
   Let $\Lambda_{H}$ be the space of predictable processes $g$ defined on $\RR_{+}\times\RR$ such that almost surely $g\in\HH$ and $\EE[\|g\|_{\HH}^2]<\infty$.  Then,  the following items hold.
   \begin{itemize}
       \item[(i).]
  The space of the elementary processes
    defined in \cite[Definition 2.2]{HHLNT2017} is dense in $\Lambda_{H}$;
      \item[(ii).]
 For any  $g\in\Lambda_{H}$, the stochastic integral $\int_{\RR_{+}}\int_{\RR} g(s,x)W(ds,dx)$ is defined  as the $L^2(\Omega)$-limit of  Riemann sums along    elementary processes approximating $g$
in $\Lambda_H$, and we have
    \begin{equation}\label{Isometry}
     \EE\lk\lc\int_{\RR_{+}}\int_{\RR} g(s,x)W(ds,dx)\rc^2\rk=\EE\big[\|g\|_{\HH}^2\big].
    \end{equation}
\end{itemize}
\end{proposition}

  Let $(B,\| \cdot \|_B)$ be a Banach space  with the norm $\| \cdot \|_B$.  Let   $H\in(\frac{1}{4},\frac{1}{2})$ be a fixed number.
 For  any  function $f:\RR\rightarrow B$,  denote
 \begin{equation}\label{NBNorm}
   \cN_{\frac{1}{2}-H}^{B}f(x):=\lt(\int_{\RR}\|f(x+h)-f(x)\|_B^2\cdot |h|^{2H-2}dh\rt)^{\frac 12},
 \end{equation}
 if the above quantity is finite.
 When $B=\RR$, we abbreviate the notation $\cN_{\frac{1}{2}-H}^{\RR}f$  as  $\cN_{\frac{1}{2}-H}f$.
When $B=L^p(\Omega)$, we  denote $ \cN_{\frac{1}{2}-H}^{B}$ by $\cN_{\frac{1}{2}-H,\,p}$.

\subsection{The fractional  heat kernel and  the linear stochastic heat equation}

Let $G_{\alpha}(t, x)$ denote the heat kernel associated to the operator $-(-\Delta)^{\frac{\alpha}{2}}$ on $\mathbb{R}$, defined via its Fourier transform
\begin{equation}\label{Eq: HeatKernel}
\cF G_{\alpha}(t, \cdot)(\xi):= e^{-t\|\xi\|^{\alpha}},\,\,\,\,\xi\in \mathbb{R}
\end{equation}
for $\alpha\in(1,2]$. Let us recall some useful properties of the kernel $\{G_{\alpha}(t,  x)\}_{t>0,x\in \mathbb R}$.
     For details, we refer to \cite{ANT2022, ChenZhang2016*}.

    It is well-known that   $G_{\alpha}(t, \cdot)$ is the probability transition density function of a    $1$-dimensional
    stable process $\{L_t^{\alpha}\}_{t\ge0}$.   By the scaling property of $L_t^{\alpha}\overset{d}{=}t^{1/\alpha} L_1^{\alpha}$,
   it follows  that
   \begin{equation}\label{eq scaling}
    G_{\alpha}(t, x)=t^{-\frac{1}{\alpha}} G_{\alpha}\big(1, t^{-\frac{1}{\alpha}}x\big) \ \ \ \  \ \ \ (t>0,\, x\in \mathbb R).
  \end{equation}
  When $\alpha=2$, $\left\{L_t^{\alpha}\right\}_{t\ge0}$  is a $1$-dimensional Brownian motion, and $G_2(t, x)=\frac{1}{(4\pi t)^{1/2}}\exp\left\{-\frac{|x|^2}{4t}\right\}$. For simplicity, we assume that $\alpha\in (1,2)$ throughout the rest of this paper. The proof in the case of $\alpha=2$ is relatively simple and thus omitted.

 When $\alpha\in (1,2)$,  by \cite[Theorem 2.1]{BG1960}, there exist some finite positive constants $c_{2,1}$ and $c_{2,2}$
 such that  for all $t>0$ and $x \in \mathbb R$,
\begin{align}\label{eq Green 3}
c_{2,1}t\left(t^{1/\alpha}+ |x| \right)^{-1-\alpha}\le G_{\alpha}(t, x)\le c_{2,2}t\left(t^{1/\alpha}+|x| \right)^{-1-\alpha},
\end{align}
which entails, for $0\le \theta<\alpha+\frac{1}{2}$,  \begin{equation}\label{eq Green 4}
\int_{\mathbb R}G_{\alpha}^2(1,z)|z|^{2\theta}dz\le c_{2,3} \int_{0}^{\infty} (1+z)^{2\theta-2\alpha-2}dz<\infty.
\end{equation}

 Hence,   by \cite[(1.17)]{ChenZhang2016*}, we know that
 there exists a positive constant $c_{2,4}>0$ such that  for all $0<s< t\le T$ and $x, y\in \mathbb R$,
\begin{align}\label{eq grad temp}
 \left|G_{\alpha}(t, x) -G_{\alpha}(s, x) \right|
\le    c_{2,4}(t-s)\cdot \left(s^{1/\alpha}+|x| \right)^{-1-\alpha}
\le  \frac{c_{2,4}}{c_{2,2}} \frac{(t-s)}{s} G_{\alpha}(s, x).
\end{align}
By \cite[Lemma 2.2]{ChenZhang2016*}, one has that for every $T > 0$, there exists a constant $c_{2,5} > 0$
such that for all $0 < t \le T $ and $x, y\in \mathbb R$,
\begin{equation}\label{eq grad spatial}
\begin{split}
 \left|G_{\alpha}(t, x) -G_{\alpha}(t, y)  \right|
\le  &\,  c_{2,5}\left( \frac{|x-y|}{t^{1/\alpha}} \wedge 1\right)\cdot \left( G_{\alpha}(t, x) +G_{\alpha}(t,  y) \right).
\end{split}
\end{equation}
\begin{proposition}\label{prop green 1}
For $\theta\in [0,\frac{1-2H}{\alpha})$, we have
\begin{equation}\label{eq green ineq1}
\int_{\mathbb R}dy\int_{\mathbb R}dz\left|G_{\alpha}(1,y)-G_{\alpha}(1,z)\right|^2|y-z|^{2H-2}|y|^{\alpha\theta}<\infty.
\end{equation}
\end{proposition}
\begin{proof}
By a change of variables $y=z+h$ and $z= z$, and using \eqref{eq grad spatial}
\begin{align*}
&\int_{\mathbb R}dy\int_{\mathbb R}dz\left|G_{\alpha}(1,y)-G_{\alpha}(1,z)\right|^2|y-z|^{2H-2}|y|^{\alpha\theta}\\
=&\int_{\mathbb R}dh\int_{\mathbb R}dz\left|G_{\alpha}(1,z+h)-G_{\alpha}(1,z)\right|^2|h|^{2H-2}|z+h|^{\alpha\theta}
\leq 2 c_{2,5}\left(I_{1}+I_{2}+I_{3}+I_{4}\right),
\end{align*}
where
\begin{align*}
I_{1}=:& \int_{\mathbb R}dh\int_{\mathbb R}dz \left(h\wedge 1\right)^2|h|^{2H-2} G_{\alpha}(1,z+h)^2|z|^{\alpha\theta},
I_{2}=:\int_{\mathbb R}dh\int_{\mathbb R}dz \left(h\wedge 1\right)^2|h|^{2H+\alpha\theta-2} G_{\alpha}(1,z+h)^2,\\
I_{3}=:&\int_{\mathbb R}dh\int_{\mathbb R}dz \left(h\wedge 1\right)^2|h|^{2H-2}G_{\alpha}(1,z)^2|z|^{\alpha\theta},
I_{4}=:\int_{\mathbb R}dh\int_{\mathbb R}dz \left(h\wedge 1\right)^2|h|^{2H+\alpha\theta-2} G_{\alpha}(1,z)^2.
\end{align*}
By a change of variables $z=z-h$ and \eqref{eq Green 4}, for $0\le\theta<\frac{1-2H}{\alpha}$, we have
\begin{align*}
I_1
\le&c_{2,6}\left[\int_{\mathbb R}\left(h\wedge 1\right)^2|h|^{2H-2}dh+\int_{\mathbb R}\left(h\wedge 1\right)^2|h|^{2H+\alpha\theta-2}dh\right]
<\infty.
\end{align*}
From the above proof, we can see immediately $I_3<\infty$. Similarly, we have
\begin{align*}
I_2=\int_{\mathbb R} \left(h\wedge 1\right)^2|h|^{2H+\alpha\theta-2} dh\int_{\mathbb R}G_{\alpha}(1,z+h)^2 dz<\infty, \ \mbox{and}\ I_4<\infty.
\end{align*}
In conclusion, the proof is now complete.\end{proof}

\subsection{Mild solution}

\begin{condition}\label{cond A} Assume that for equation \eqref{SFHE} the following conditions hold:
 \begin{itemize}
 \item[(A1).] For $H\in (\frac{3-\alpha}{4},  \frac{1}{2})$ and for some $\beta_0> \frac{1}{2}-H$ and some   $p_0>\max \left(\frac{2(\alpha+1)}{\alpha+4H-3}, \frac{1}{\beta_0+H-{1}/{2}}\right)$, the initial condition $u_0$ is in $L^{p_0}(\RR)\cap L^\infty(\RR)$ and satisfies
 \begin{equation}\label{eq217}
   \sup_{x\in\RR}\mathcal{N}_{\beta_0}u_0(x)+ \left(\int_{\RR}\|u_0(\cdot)-u_0(\cdot+h)\|^2_{L^{p_0}(\RR)}|h|^{2H-2}dh\right)^{\frac12} < \infty\,.
 \end{equation}
 \item[(A2).] $\sigma$ is differentiable, the derivative of $\sigma$ is Lipschitz and $\sigma(0)=0$.
 \end{itemize}
\end{condition}

\begin{theorem}\cite[Theorems 2 and 3]{LM2022}\label{them existence}
 Under Condition \ref{cond A}, for any $p\geq p_0$, there exists a solution $u$ to \eqref{SFHE}  satisfying the following moment bounds:
    \begin{equation}\label{eq bound}
\sup_{ t\in[0,T],x\in \mathbb R}\|u(t,x)\|_{L^p(\Omega)}<\infty,\ \ \  \ \ \ \sup_{t\in[0,T],x\in \mathbb R}\mathcal N_{1/2-H, p} u(t,x)<\infty.
\end{equation}
 In addition, if   the initial condition $u_0$ is H\"older continuous with order $\beta_0$, then we have
 \begin{equation}\label{eq Holder}
\|u(t,x)-u(s,y)\|_{L^p(\Omega)} \leq c_{2,7} \left(|t-s|^{\vartheta_0}+ |x-y|^{\alpha\vartheta_0}\right),
\end{equation}
for all $s, t \in [0,T]$ and $x,y \in \RR$ with $\vartheta_0:=\frac{\alpha+2H-2}{2\alpha}\wedge \frac {\beta_0}{\alpha}$.
   \end{theorem}

\begin{remark}\label{remark1}
By Proposition A.1 in \cite{HHLNT2017}, Condition \ref{cond A} ensures that the initial function $u_0$ is H\"older continuous with order $\beta_0$. By Remark 2.1 in \cite{QianWAngXiao2026}, Condition \ref{cond A} automatically implies that \eqref{eq217} holds for all $p\geq p_0$.
\end{remark}

  \begin{proposition}\label{HBDG}   $($\cite[Proposition 3.2]{HHLNT2017}$)$
   Let $W$ be the Gaussian noise with the covariance \eqref{CovStru}, and let    $f\in\Lambda_H$  be a predictable random field. Then, we have   {that} for any $p\geq2$,
   \begin{equation}\label{eq BDG}
     \begin{split}
        \lt\|\int_{0}^{t}\int_{\RR}f(s,y) W(ds,dy)\rt\|_{L^p(\Omega)}
     \leq \sqrt{4p}c_{2,8}\lt(\int_{0}^{t}\int_{\RR}\lk\cN_{\frac 12-H,\,p}f(s,y)\rk^2dyds\rt)^{\frac 12},
     \end{split}
   \end{equation}
   where $c_{2,8}$ is a constant depending only on $H$ and
   $\cN_{\frac 12-H,\,p}f(s,y)$ denotes the application of $\cN_{\frac 12-H,\,p} $  to the spatial variable $y$.
 \end{proposition}

\section{The main results and applications}\label{sec main}

Our main result is the following estimation on the moments of the approximation error of the temporal increments. For any $\e>0$ and any random field $\{X(t,x)\}_{t\geq0,x\in\mathbb{R}}$, denote
$(\mathcal{D}_{\e}X)(t,x):=X(t+\e,x)-X(t,x)$ for $t\geq0$  and $x\in\mathbb{R}$. Let $\vartheta_0:=\frac{\alpha+2H-2}{2\alpha}\wedge\frac{\beta_0}{\alpha}$ and
 \begin{align*}
 \mathcal{I}:=\left(0, \,\frac{(\alpha-2H+2)\vartheta_0}{2\alpha(1+\vartheta_0)}\right)\cap\left(0, \,\frac{(\alpha-2H+2)(1-2H)}{2\alpha(2\alpha-2H+1)}\right).
 \end{align*}

  \begin{theorem}\label{thm main} Assume that Condition \ref{cond A} holds. Then for every $p\geq 1$ and $\eta\in \mathcal{I}$, there exists a finite constant  $c_{3,1}>0$
   such that  for all $\varepsilon\in(0,1)$,  $t\in [0,T]$ and $x\in \mathbb R$,
     \begin{align}\label{eq main}
    &\Big\| (\mathcal{D}_{\e}u)(t,x)-\big[ G_{\alpha}(t+\varepsilon,\cdot)-G_{\alpha}(t,\cdot)\big]*u_0(x)-\sigma(u(t, x))(\mathcal{D}_{\e}v)(t,x) \Big\|_{L^p(\Omega)}
    \le  c_{3,1}\,   \varepsilon^{\frac{\alpha+2H-2}{2\alpha}+\eta}.
     \end{align}

    In addition,     if the initial condition $u_0$ is H\"older continuous with order $ \beta_0>\frac{\alpha+2H-2}{\alpha}$, then there exists a finite constant  $c_{3,2}>0$
   such that  for all $\varepsilon\in(0,1)$,
     \begin{align}\label{eq main 2}
   \sup_{t\in[0,T]}\sup_{x\in \mathbb{R}} \Big\|(\mathcal{D}_{\e}u)(t,x)-\sigma(u(t, x))(\mathcal{D}_{\e}v)(t,x) \Big\|_{L^p(\Omega)}\le c_{3,2}\varepsilon^{\frac{\alpha+2H-2}{2\alpha}+\eta}.
     \end{align}
     \end{theorem}

By applying  Theorem \ref{thm main}  and  the similar arguments in \cite{FKM2015,  WX2024},
we  obtain the following results.

   \begin{corollary}\label{cor  t LIL}  Assume  Condition \ref{cond A} holds.
    Choose and fix $t>0$ and $x\in \mathbb R$.  Then with probability one,
    \begin{itemize}
    \item[(a).] (Khinchin's law of the iterated logarithm)
    \begin{equation}\label{Eq:LIL}
    \begin{split}
    \limsup_{\e \to 0}\frac{|u(t+\varepsilon, x)-u(t, x)|}{\e^{\widetilde{H}/2}\sqrt{2 \log\log(1/\e)}}= \kappa |\sigma(u(t, x))|,
    \end{split}
    \end{equation}
    where $\widetilde{H}=\frac{\alpha+2H-2}{\alpha}$.
        \item[(b).]  (Chung's law  of the iterated logarithm)
    \begin{equation}\label{Eq: CLIL}
    \begin{split}
    \liminf_{\e \to 0}\sup_{0\le r\le \e}\frac{|u(t+r, x)-u(t, x)|}{(\e/  \log\log(1/\e))^{\widetilde{H}/2}}=  \kappa\lambda_{\widetilde{H}}^{\widetilde{H}/2} |\sigma(u(t, x))|,
    \end{split}
    \end{equation}
    where $\lambda_{\widetilde{H}}$ is the small ball  constant of a fBm with index $\widetilde{H}/2$ (see, e.g., \cite[Theorem 6.9]{LS01*}).
  \end{itemize}
  If   $ \beta_0>\frac{\alpha+2H-2}{\alpha}$ in Condition \ref{cond A}, then both results also hold at $t=0$ with $\kappa$ in \eqref{eq constant 1} replaced by   \begin{align}\label{eq con kappa1}
             \widetilde\kappa= \left( \frac{ \Gamma(2\widetilde{H}) }{2^{1-\widetilde{H}} \Gamma(\widetilde{H})}\right)^{1/2}.
             \end{align}
         \end{corollary}

\section{Proofs of main results}
This section is devoted to the proof of the temporal increment approximation result.
The argument is based on the approximation techniques developed in
\cite{KSXZ2013,WZ2021*,QianWAngXiao2026}, with suitable modifications to handle the fractional diffusion operator and the rough spatial covariance structure. Throughout this section, we work under Condition \ref{cond A}.

\subsection{An approximation theorem for the stochastic integral}
    For fixed $x\in\mathbb R$, the temporal increment of the mild solution can be written as $u(t+\varepsilon, x)-u(t, x)  =\, \mathcal{J}_0+ \mathcal J_1+\mathcal J_2$,
  where
  \begin{align}
  \mathcal{J}_0:=&\,\int_{\mathbb{R}}\left(G_{\alpha}(t+\varepsilon,x-y)-G_{\alpha}(t,x-y)\right)u_0(y)dy, \label{eq J0}\\
  \mathcal J_1:=&\, \int_0^t \int_{\mathbb R} \, \left[G_{\alpha}(t+\varepsilon-s,x-y)- G_{\alpha}(t-s,x-y)\right] \sigma(u(s, y))W(ds,dy),
  \label{eq J1}\\
  \mathcal J_2:=&\, \int_t^{t+\varepsilon} \int_{\mathbb R} \,G_{\alpha}(t+\varepsilon-s,x-y) \sigma(u(s, y))W(ds,dy).\label{eq J2}
  \end{align}

 We first consider the approximation of $\mathcal J_{2}$. Define
  \begin{align}
  \widetilde{\mathcal J_1}&:=\, \sigma(u(t, x))\int_0^t \int_{\mathbb R} \, \left[G_{\alpha}(t+\varepsilon-s,x-y)- G_{\alpha}(t-s,x-y)\right] W(ds,dy),\label{eq tilde J1}\\
  \widetilde{\mathcal J_{2}}&:=\,\sigma(u(t, x))\int_t^{t+\varepsilon} \int_{\mathbb R} \,G_{\alpha}(t+\varepsilon-s,x-y)  W(ds,dy).\label{eq tilde J2}
  \end{align}
  \begin{lemma}\label{lem J2}
  For any  $T>0$,  $p>p_0$,  $\vartheta\le \vartheta_0$ and $\vartheta<\frac{1-2H}{\alpha}$,  there exists a finite positive constant  $c_{3,3}$, independent of  $\varepsilon\in (0,1)$, such that
     \begin{align*}\label{eq J2 diff1}
  \sup_{t\in [0,T]} \sup_{x\in \mathbb R}  \left\|\mathcal J_2-\widetilde{\mathcal J_2} \right\|_{L^p(\Omega)} ^2 \le c_{3,3} \varepsilon^{ \frac{\alpha +2H-2}{\alpha} + \vartheta }.
  \end{align*}
  \end{lemma}
 \begin{proof} Set $g_{\varepsilon, t,  x}(s,y):=  G_{\alpha}(t+\varepsilon-s, x-y) \left[\sigma\big(u(s, y)\big)- \sigma\big(u(t, x)\big)\right]$. Then  $\mathcal J_2-\widetilde{\mathcal J_2}=\int_t^{t+\varepsilon} \int_{\mathbb R} g_{\varepsilon, t, x}(s,y) W(ds,dy)$.  Applying   the inequality  \eqref{eq BDG},  we have
 \begin{equation*}
\begin{split}
 \left\|\mathcal J_2-\widetilde{\mathcal J_2} \right\|_{L^p(\Omega)}^2
\le  \,  4p c_{2,8}^2 \int_t^{t+\varepsilon}ds \int_{\mathbb R}dy  \int_{\mathbb R}dz   \left\|   g_{\varepsilon, t,  x}(s,y) -  g_{\varepsilon, t,  x}(s,z)  \right\|^2_{L^p(\Omega)}  |y-z|^{2H-2}.
 \end{split}
\end{equation*}
 By separating the increment of the heat kernel from that of the nonlinear coefficient, we have
\begin{equation} \label{eq A 1}
  \left\|\mathcal J_2-\widetilde{\mathcal J_2} \right\|_{L^p(\Omega)}^2 \le   8p c_{2,8}^2 \left[  \int_t^{t+\varepsilon}  I_{1}(s)ds + \int_t^{t+\varepsilon}  I_{2}(s)ds\right],
\end{equation}
 where
\begin{align*} 
 I_{1}(s):=&\, \int_{\mathbb R}dy  \int_{\mathbb R}dz   \big|G_{\alpha}(t+\varepsilon-s,  x-y)- G_{\alpha}(t+\varepsilon-s,  x-z)  \Big|^2|y-z|^{2H-2} \left\|\sigma\big(u(s, y)\big)- \sigma\big(u(t, x)\big)\right\|_{L^p(\Omega)}^2,\\
 I_{2}(s):= &\, \int_{\mathbb R}dy  \int_{\mathbb R}dz G_{\alpha}(t+\varepsilon-s, x-z)^2
   |y-z|^{2H-2} \left\|\sigma\big(u(s, y)\big)- \sigma\big(u(s, z)\big)\right\|_{L^p(\Omega)} ^2.
 \end{align*}

For $I_{1}(s)$, by using \eqref{eq bound}, \eqref{eq Holder} and Remark \ref{remark1}, there exists a positive constant $c_{3,4}$ such that for any $\vartheta\le \vartheta_0$,
 \begin{align} \label{eq A0 112}
 I_{1}(s) \le &\, c_{3,4}
   \int_{\mathbb R}dy  \int_{\mathbb R}dz   \big|G_{\alpha}(t+\varepsilon-s,  x-y)- G_{\alpha}(t+\varepsilon-s,  x-z)  \Big|^2 |y-z|^{2H-2} \left[\big(s-t\big)^{\vartheta}+ |y-x|^{\alpha\vartheta}\right]\notag\\
= &\,
c_{3,4}   \big(s-t\big)^{\vartheta}
  \int_{\mathbb R}dy  \int_{\mathbb R}dz   \big|G_{\alpha}(t+\varepsilon-s,  x-y)- G_{\alpha}(t+\varepsilon-s,  x-z)  \Big|^2     |y-z|^{2H-2}   \notag\\
&\,+ c_{3,4}   \int_{\mathbb R}dy  \int_{\mathbb R}dz   \big|G_{\alpha}(t+\varepsilon-s,  x-y)- G_{\alpha}(t+\varepsilon-s,  x-z)  \Big|^2  |y-z|^{2H-2}  |y-x|^{\alpha\vartheta}\notag\\
=:&\, c_{3,4}\left[I_{1, 1}(s)+ I_{1, 2}(s)\right].
 \end{align}
By  Lemma \ref{lem int p} ($\beta=\frac12-H$), there exists a positive constant $c_{3,5}$ such that
 \begin{align} \label{eq A 11-10}
  \int_t^{t+\varepsilon}   I_{1, 1}(s)ds \le   c_{3,5}   \int_t^{t+\varepsilon} (s-t)^{\vartheta} (t+\varepsilon-s)^{\frac{2H-2}{\alpha}} ds
 \le   c_{3,5} \mathcal B\left(\vartheta+1, \frac{\alpha +2H-2}{\alpha}\right)  \varepsilon^{\vartheta+\frac{\alpha +2H-2}{\alpha}}.
  \end{align}
 Here, $\mathcal B(a,b)$ is the Beta function. By a  change of variables,  \eqref{eq scaling} and \eqref{eq green ineq1},  there exists a positive constant c such that for any $\vartheta<\frac{1-2H}{\alpha}$,
 \begin{align*} 
&  \int_{\mathbb R}dy  \int_{\mathbb R}dz   \big|G_{\alpha}(t+\varepsilon-s,  x-y)- G_{\alpha}(t+\varepsilon-s,  x-z)  \big|^2  |y-z|^{2H-2}  |y-x|^{\alpha\vartheta}\\
=&\,  (t+\varepsilon-s)^{\vartheta+\frac{2H-2}{\alpha}}\int_{\mathbb R}dy  \int_{\mathbb R}dz   \big|G_{\alpha}(1,  y)- G_{\alpha}(1,  z)  \Big|^2  |y-z|^{2H-2}  |y|^{\alpha\vartheta}
\le c_{3,6}(t+\varepsilon-s)^{\vartheta+\frac{2H-2}{\alpha}}.
\end{align*}
Thus,
  \begin{equation} \label{eq A  11-2}
\int_t^{t+\varepsilon}   I_{1,2}(s)ds \le\,c_{3,6} \int_t^{t+\varepsilon} (t+\varepsilon-s)^{\vartheta+\frac{2H-2}{\alpha}}   ds = \frac{c_{3,6}}{\vartheta+\frac{\alpha+2H-2}{\alpha}} \varepsilon^{\vartheta+\frac{\alpha+2H-2}{\alpha}}.
 \end{equation}

For $I_{2}(s)$, since the derivative of $\sigma$ is Lipschitz, by \eqref{eq bound}, we have
$$
\sup_{s\in [0,T], x\in \mathbb R} \left[\mathcal N_{\frac{1}{2}-H, p} \sigma(u(s,x))\right]^2\le\, L_{\sigma}^2\sup_{s\in [0,T], x\in \mathbb R} \left[\mathcal N_{\frac{1}{2}-H, p} u(s,x)\right]^2<\infty,
$$
where $L_{\sigma}$ is the Lipschitz constant of $\sigma$. It follows  that
\begin{align}\label{eq A 12-2}
\int_t^{t+\varepsilon}  I_{2}(s)ds
  =c_{3,7}  \int_t^{t+\varepsilon} (t+\varepsilon-s)^{-\frac1{\alpha}}ds \le c_{3,7}\frac{\alpha}{\alpha-1} \varepsilon^{\frac{\alpha-1}{\alpha}}.
\end{align}

The desired result follows from \eqref{eq A 1}, \eqref{eq A0 112}, \eqref{eq A 11-10}, \eqref{eq A 11-2} and \eqref{eq A 12-2}. This completes the proof. \end{proof}

We next consider the approximation of $\mathcal J_1$.
  Let $\theta\in (0,1)$, and denote  $\mathcal J_1:=\mathcal J_{1, \theta}+\mathcal J_{1, \theta}'$ and $\widetilde{\mathcal J}_1:=\widetilde{\mathcal J}_{1, \theta}+\widetilde{\mathcal J}_{1, \theta}'$,
  where
  \begin{equation}\label{LAbJ1theta}
  \begin{split}
   \mathcal J_{1, \theta}:=&\, \int_0^{t(\varepsilon)}\int_{\mathbb R}\left[G_{\alpha}(t+\varepsilon-s,x-y)-G_{\alpha}(t-s,x-y)\right]\sigma(u(s, y))W(ds,dy),\\
      \mathcal J_{1, \theta}':=&\, \int_{t(\varepsilon)}^t\int_{\mathbb R}\left[G_{\alpha}(t+\varepsilon-s,x-y)-G_{\alpha}(t-s,x-y)\right]\sigma(u(s, y))W(ds,dy),\\
       \widetilde{\mathcal J}_{1, \theta}:=&\,\sigma(u(t,x)) \int_0^{t(\varepsilon)}\int_{\mathbb R}\left[G_{\alpha}(t+\varepsilon-s,x-y)-G_{\alpha}(t-s,x-y)\right]W(ds,dy),\\
       \widetilde{\mathcal J}_{1, \theta}':=&\, \sigma(u(t, x))\int_{t(\varepsilon)}^t\int_{\mathbb R}\left[G_{\alpha}(t+\varepsilon-s,x-y)-G_{\alpha}(t-s,x-y)\right]W(ds,dy).
  \end{split}
  \end{equation}
  Here, $t(\varepsilon):=t-\varepsilon^{\theta}$ for any $t>0$ and $\varepsilon$ small enough such that $t(\e)>0$.

  We adopt a modified version of the approximation argument developed in
\cite{KSXZ2013, WX2024} to estimate $\mathcal J_1$.
\begin{itemize}
\item
First, we show that the contribution of $\mathcal J_{1,\theta}$ is of higher order for a suitable choice of $\theta\in(0,1)$.

\item
Second, since the kernels
$G_{\alpha}(t+\varepsilon-s,x-y)$ and
$G_{\alpha}(t-s,x-y)$ become increasingly concentrated as $s\uparrow t$
for $s\in(t(\varepsilon),t)$, we show that
$\mathcal J_{1,\theta}'$ can be approximated by $\mathcal J_{1,\theta}''$, where
\begin{align}\label{eq appr 2}
\mathcal J_{1,\theta}''
:=&\,\sigma\left(u(t(\varepsilon),x)\right)
\int_{t(\varepsilon)}^t\int_{\mathbb R}
\left[
G_{\alpha}(t+\varepsilon-s,x-y)-G_{\alpha}(t-s,x-y)
\right]
W(ds,dy).
\end{align}

\item
Finally, using the H\"older continuity of $u$, we show that
$\mathcal J_{1,\theta}''$ can be approximated by
$\widetilde{\mathcal J}_{1,\theta}'$.
\end{itemize}
 In Lemmas \ref{lem J11}-\ref{lem J14} below, we will prove that the errors of  those approximations  remain sufficiently small for our needs.

    For any $t, \varepsilon>0$,  denote
    \begin{align}\label{eq differ2}
          G_{\alpha}([t,\,  t+\varepsilon],x):=G_{\alpha}(t+\varepsilon,x)-G_{\alpha}(t,x) \ \ \ \ \ \ \ \ \ (t\ge 0, x\in \mathbb R).
         \end{align}

 \begin{lemma}\label{lem J11} For every  $T>0$, $\theta\in(0,1)$ and $p>p_0$,  there exists a finite positive  constant $c_{3,8}$, which does not depend on
  $\varepsilon\in (0,1)$, such that
  \begin{align*}
 \sup_{t\in [0,T]}  \sup_{x\in \mathbb R}  \left\|\mathcal J_{1, \theta} \right\|_{L^p(\Omega)}^2  \le  c_{3,8}\varepsilon^{\frac{\alpha+2H-2}{\alpha}+\left[\frac{\alpha-2H+2}{\alpha}(1-\theta)\right]\wedge  \frac{1-2H}{\alpha} }.
  \end{align*}
  \end{lemma}

\begin{proof}
     Set $g_{t,  x, \varepsilon}(s,y) :=  G_{\alpha}([t-s, \, t+\varepsilon-s],x-y)\sigma\big(u(s, y)\big)$. Then $\mathcal J_{1, \theta} =\int_0^{t(\varepsilon)}  \int_{\mathbb R} g_{t, x, \varepsilon}(s,y) W(ds,dy)$. Applying  BDG's inequality  \eqref{eq BDG},  we have
 \begin{equation*} \label{eq A 3}
\begin{split}
\left\|\mathcal J_{1, \theta} \right\|_{L^p(\Omega)}^2
\le  \,  4pc_{2,8}^2 \int_0^{t(\varepsilon)} ds \int_{\mathbb R}dy  \int_{\mathbb R}dz   \left\|   g_{t,  x, \varepsilon}(s,y) -  g_{t,  x, \varepsilon}(s,z)  \right\|^2_{L^p(\Omega)}  |y-z|^{2H-2}.
 \end{split}
\end{equation*}
 By separating the increment of the heat kernel from that of the nonlinear coefficient, we have
\begin{equation}\label{eq A decomp}
\left\|\mathcal J_{1, \theta} \right\|_{L^p(\Omega)}^2 \le 8p c_{2,8}^2  \left[  \int_0^{t(\varepsilon)}  J_{1}(s)ds + \int_0^{t(\varepsilon)}  J_{2}(s)ds\right],
 \end{equation}
 where
\begin{align*}
 J_{1}(s):=&\, \int_{\mathbb R}dy  \int_{\mathbb R}dz  \Big| G_{\alpha}([t-s,\, t+\varepsilon-s],x-y)- G_{\alpha}([t-s,\, t+\varepsilon-s], x-z)  \Big|^2 |y-z|^{2H-2} \left\|\sigma\big(u(s, y)\big) \right\|_{L^p(\Omega)} ^2,\\
 J_{2}(s):=& \, \int_{\mathbb R}dy  \int_{\mathbb R}dz\,  G_{\alpha}([t-s,\, t+\varepsilon-s], x-z)^2
   |y-z|^{2H-2} \left\|\sigma\big(u(s, y)\big)- \sigma\big(u(s, z)\big)\right\|_{L^p(\Omega)} ^2.
 \end{align*}

For $J_{1}(s)$, by   \eqref{eq bound}, a change of variables, Plancherel's identity and the elementary inequality $0\le 1-e^{-x}\le 1\wedge x\,(x\ge0)$, we have
\begin{align} \label{eq A 310}
 \int_0^{t(\varepsilon)}  J_{1}(s) ds
 \le  &\, c_{3,9}\int_0^{t(\varepsilon)}ds    \int_{\mathbb R}dy  \int_{\mathbb R}dh  \Big| G_{\alpha}([t-s,\, t+\varepsilon-s], y)- G_{\alpha}([t-s,\, t+\varepsilon-s],y+h)  \Big|^2
  |h|^{2H-2}\notag \\
= &\,  c_{3,9}\int_0^{t(\varepsilon)}ds   \int_{\mathbb R}d\xi    \, e^{-(t-s)  |\xi|^{\alpha}} \left|e^{- \varepsilon|\xi|^{\alpha}} -1\right|^2   |\xi|^{1-2H} \cdot \int_{\mathbb R}  \left| e^{-i  h}-1\right|^2  |h|^{2H-2} dh\notag\\
\le &\,  c_{3,9} \varepsilon^{\frac{\alpha+2H-2}{\alpha}+\frac{\alpha-2H+2}{\alpha}(1-\theta)}\int_{\mathbb{R}}e^{-|\xi|^{\alpha}}|\xi|^{2\alpha+1-2H}d\xi\cdot \int_{\mathbb R}  \left| e^{-i  h}-1\right|^2  |h|^{2H-2} dh\notag\\
\le &c_{3,10}\varepsilon^{\frac{\alpha+2H-2}{\alpha}+\frac{\alpha-2H+2}{\alpha}(1-\theta)}.
\end{align}

For $J_{2}(s)$,  using Plancherel's identity, we  have
\begin{align} \label{eq A 320}
 \int_0^{t(\varepsilon)}  J_{2}(s) ds
 \le &\, \sup_{s\in [0,T],\,  x\in \mathbb R}  \left[\mathcal N_{\frac{1}{2}-H, p} \sigma(u(s,x)) \right]^2    \int_0^{t(\varepsilon)}ds   \int_{\mathbb R}dz\,  G_{\alpha}([t-s,\, t+\varepsilon-s],x-z)^2    \nonumber\\
\le &\,    c_{3,11} \int_0^{t(\varepsilon)} ds   \int_{\mathbb R}  d\xi  \,   e^{-2(t-s)  |\xi|^{\alpha}} \left|e^{-  \varepsilon |\xi|^{\alpha}} -1\right|^2 \nonumber \\
= &\,  \frac{ c_{3,11}}{2}    \int_{\mathbb R}  \left[e^{-2\e^{\theta}|\xi|^{\alpha}}-e^{-2t|\xi|^{\alpha}} \right]   \left|e^{- \varepsilon   |\xi|^{\alpha}}-1  \right|^2 |\xi|^{-\alpha}d\xi
\le \frac{c_{3,11}}{2} \varepsilon^{\frac{\alpha-1}{\alpha}}.
 \end{align}

The desired result follows from \eqref{eq A decomp}, \eqref{eq A 310} and \eqref{eq A 320}. This completes the proof.
  \end{proof}

\begin{corollary}\label{cor tildeJ1 theta}For every  $T>0$, $\theta\in(0,1)$ and $p>p_0$, there exists a finite positive  constant $c_{3,12}$, which does not depend on
  $\varepsilon\in (0,1)$, such that
 \begin{align*}
 \sup_{t\in [0,T]}  \sup_{x\in \mathbb R}  \left\|\widetilde{\mathcal J}_{1, \theta} \right\|_{L^p(\Omega)}^2  \le  c_{3,12}\varepsilon^{\frac{\alpha+2H-2}{\alpha}+\frac{\alpha-2H+2}{\alpha}(1-\theta) }.
  \end{align*}

\end{corollary}

    \begin{lemma}\label{lem J12} For  every  $T>0$,  $\theta\in (0,1)$, $p>p_0$, $0\le \vartheta< \frac{1-2H}{2\alpha}$ and $\vartheta\le \vartheta_0$, there exists a
     finite positive constant $c_{3,13}$ such that  for all  $\varepsilon\in (0,1)$,
          \begin{align*}
           \sup_{t\in [0,T]} \sup_{x\in \mathbb R}\left\|\mathcal J_{1, \theta}'-\mathcal J_{1, \theta}'' \right\|_{L^p(\Omega)}^2 \le c_{3,13} \varepsilon^{\frac{\alpha-1}{\alpha}\wedge \left[\theta\left(2\vartheta+\frac{\alpha+2H-2}{\alpha}\right)\right]}.
              \end{align*}

          \end{lemma}
          \begin{proof}
          Since $u\big(t(\e), x\big)\in \mathcal F_{t(\e)}$,
     $\sigma\big(u(t(\e), x)\big)$ can be put  inside the stochastic integral $\mathcal J_{1, \theta}''$. Set $
   g_{t,  x, \varepsilon}(s,y) :=  G_{\alpha}([t-s,\, t+\varepsilon-s],x-y) \left[\sigma\big(u(s, y)\big)- \sigma\big(u(t(\varepsilon), x)\big)\right]$. Then $\mathcal J_{1, \theta}'-\mathcal J_{1, \theta}''=\int_{t(\varepsilon)}^{t } \int_{\mathbb R} g_{t, x, \varepsilon}(s,y) W(ds,dy)$.  Applying BDG's inequality  \eqref{eq BDG},  we have

 \begin{equation*} \label{eq A 2}
\begin{split}
 \left\|\mathcal J_{1, \theta}'-\mathcal J_{1, \theta}'' \right\|_{L^p(\Omega)}^2
\le  \,  4pc_{2,8}^2 \int_{t(\varepsilon)}^t ds \int_{\mathbb R}dy  \int_{\mathbb R}dz   \left\|   g_{t,  x, \varepsilon}(s,y) -  g_{t,  x, \varepsilon}(s,z)  \right\|^2_{L^p(\Omega)}  |y-z|^{2H-2}.
 \end{split}
\end{equation*}
 By separating the increment of the heat kernel from that of the nonlinear coefficient, we have
 \begin{equation}\label{eq A decomp1}
 \left\|\mathcal J_{1, \theta}'-\mathcal J_{1, \theta}'' \right\|_{L^p(\Omega)}^2\le  8pc_{2,8}^2\left[  \int_{t(\varepsilon)}^t  K_{1}(s)ds + \int_{t(\varepsilon)}^t  K_{2}(s)ds\right],
\end{equation}
 where
\begin{align*} \label{eq A 11}
 K_{1}(s):=&\, \int_{\mathbb R}dy  \int_{\mathbb R}dz  \Big| G_{\alpha}([t-s,\, t+\varepsilon-s],x-y)- G_{\alpha}([t-s,\, t+\varepsilon-s], x-z)  \Big|^2\\
&  \ \,\,\,  \cdot  |y-z|^{2H-2} \left\|\sigma\big(u(s, y)\big)- \sigma\big(u(t(\varepsilon), x)\big)\right\|_{L^p(\Omega)} ^2,\\
 K_{2}(s):= &\, \int_{\mathbb R}dy  \int_{\mathbb R}dz\,  G_{\alpha}([t-s,\, t+\varepsilon-s],\,x-z)^2
   |y-z|^{2H-2} \left\|\sigma\big(u(s, y)\big)- \sigma\big(u(s, z)\big)\right\|_{L^p(\Omega)} ^2.
 \end{align*}

For $ K_{1}(s)$, by   \eqref{eq Holder} and Remark \ref{remark1}, there exists a positive constant $c_{3,14}$ that for any $\vartheta\le \vartheta_0$,
    \begin{equation*} \label{eq A 111}
\begin{split}
 K_{1}(s)
  \le&\  c_{3,14}   \big(s-t(\varepsilon)\big)^{2\vartheta}
 \int_{\mathbb R}dy  \int_{\mathbb R}dz    \Big| G_{\alpha}([t-s,\, t+\varepsilon-s],x-y)- G_{\alpha}([t-s,\, t+\varepsilon-s], x-z)  \Big|^2|y-z|^{2H-2} \\
&\,+ c_{3,14}   \int_{\mathbb R}dy  \int_{\mathbb R}dz   \Big| G_{\alpha}([t-s,\, t+\varepsilon-s],x-y)- G_{\alpha}([t-s,\, t+\varepsilon-s], x-z)  \Big|^2 |y-z|^{2H-2}  |y-x|^{2\alpha\vartheta}\\
=:&\,c_{3,14}\left[ K_{1, 1}(s)+ K_{1, 2}(s)\right].
\end{split}
\end{equation*}
   Since $|s-t(\varepsilon)|\le \varepsilon^{\theta}$ for any $s\in [t(\varepsilon), t]$,  by Plancherel's identity and a change of variables, we have
 \begin{align} \label{eq A 112}
&\int_{t(\varepsilon)}^t  K_{1,1}(s)ds\notag\\
\le &\,  \varepsilon^{2\theta \vartheta} \int_{t(\varepsilon)}^t ds   \int_{\mathbb R}d\xi  \int_{\mathbb R}dh\,  e^{-2(t-s)|\xi|^{\alpha}} \left|e^{- \varepsilon |\xi|^{\alpha}}- 1 \right|^2  \left| e^{-i \xi h}-1 \right|^2|h|^{2H-2}\notag\\
\le &\,  \varepsilon^{2\theta \vartheta} \int_{t(\varepsilon)}^t ds   \int_{\mathbb R}d\xi      e^{-2(t-s)|\xi|^{\alpha}} \left|e^{- \varepsilon |\xi|^{\alpha}}- 1 \right|^2 |\xi|^{1-2H} \cdot  \int_{\mathbb R} \left| e^{-i  h}-1 \right|^2 |h|^{2H-2} dh \notag \\
\le &\,\frac{1}{2}\varepsilon^{\frac{\alpha+2H-2}{\alpha}+2\theta \vartheta}  \int_{\mathbb R}d\xi     \left|1- e^{-   |\xi|^{\alpha}} \right|^2 |\xi|^{1-2H-\alpha} \cdot  \int_{\mathbb R} \left| e^{-i  h}-1 \right|^2 |h|^{2H-2} dh
  \le \,c_{3,15} \varepsilon^{\frac{\alpha+2H-2}{\alpha}+2\theta \vartheta}.
 \end{align}
  By a change of variables, we have
\begin{align*}
K_{1,2}(s)
&\le 2\int_{\mathbb{R}}dy\int_{\mathbb{R}}dh\left|G_{\alpha}(t+\varepsilon-s,y+h)-G_{\alpha}(t+\varepsilon-s,y)\right|^2|h|^{2H-2}|y|^{2\alpha\vartheta}\\
&\hskip12pt+ 2\int_{\mathbb{R}}dy\int_{\mathbb{R}}dh\left|G_{\alpha}(t-s,y+h)-G_{\alpha}(t-s,y)\right|^2|h|^{2H-2}|y|^{2\alpha\vartheta}.
\end{align*}
By \eqref{eq scaling}, a change of variables and Proposition \ref{prop green 1}, we have for any $0\le \vartheta< \frac{1-2H}{2\alpha}$,
\begin{align*}
&\int_{\mathbb{R}}dy\int_{\mathbb{R}}dh\left|G_{\alpha}(t+\varepsilon-s,y+h)-G_{\alpha}(t+\varepsilon-s,y)\right|^2|h|^{2H-2}|y|^{2\alpha\vartheta}\\
=&(t+\varepsilon-s)^{-\frac{2}{\alpha}}\int_{\mathbb{R}}dy\int_{\mathbb{R}}dh\left|G_{\alpha}(1,(t+\varepsilon-s)^{-\frac1{\alpha}}(y+h))-G_{\alpha}(1,(t+\varepsilon-s)^{-\frac1{\alpha}}y)\right|^2
|h|^{2H-2}|y|^{2\alpha\vartheta}\\
=&(t+\varepsilon-s)^{2\vartheta+\frac{2H-2}{\alpha}}\int_{\mathbb{R}}dy\int_{\mathbb{R}}dh\left|G_{\alpha}(1,y+h)-G_{\alpha}(1,y)\right|^2|h|^{2H-2}|y|^{2\alpha\vartheta}
\le c_{3,16} (t+\varepsilon-s)^{2\vartheta+\frac{2H-2}{\alpha}}.
\end{align*}
Hence for any $0\le \vartheta< \frac{1-2H}{2\alpha}$ with $\vartheta\le \vartheta_0$, $K_{1,2}(s)\le2c_{3,16}\left((t+\varepsilon-s)^{2\vartheta+\frac{2H-2}{\alpha}}+(t-s)^{2\vartheta+\frac{2H-2}{\alpha}}\right)$.
Thus, we have for any for any $0\le \vartheta< \frac{1-2H}{2\alpha}$ with $\vartheta\le \vartheta_0$ and any $0<\theta<1$,
 \begin{align*} \label{eq A 111-2}
\int_{t(\varepsilon)}^t  K_{1,2}(s)ds &\le 2c_{3,16}\left( \left(\varepsilon+\varepsilon^{\theta}\right)^{2\vartheta+\frac{\alpha+2H-2}{\alpha}}+\varepsilon^{\theta\left(2\vartheta+\frac{\alpha+2H-2}{\alpha}\right)}\right)
\le c_{3,17}\varepsilon^{\theta\left(2\vartheta+\frac{\alpha+2H-2}{\alpha}\right)},
\end{align*}
where we use the fact that $\alpha+2H-2>0$.
Combining this with \eqref{eq A 112}, we find that for any  $\theta\in(0,1)$,
\begin{equation} \label{eq A 11-1}
\begin{split}
 \int_{t(\varepsilon)}^t K_{1}(s) ds\le     \,  c_{3,18}  \varepsilon^{\theta\left(2\vartheta+\frac{\alpha+2H-2}{\alpha}\right)}.
 \end{split}
\end{equation}

 For $ K_{2}(s)$, using the same argument as that in \eqref{eq A 320},  we have
\begin{equation} \label{eq A 22-2}
\begin{split}
 \int_{t(\varepsilon)}^t K_{2}(s) ds\le    \, c_{3,19}\varepsilon^{\frac{\alpha-1}{\alpha}}.
 \end{split}
\end{equation}

Putting \eqref{eq A decomp1}, \eqref{eq A 11-1} and \eqref{eq A 22-2} together, we conclude that for any $0\le \vartheta< \frac{1-2H}{2\alpha}$, with  $\vartheta\le \vartheta_0$ and any $0<\theta<1$,
\begin{align*}
 \left\|\mathcal J_{1, \theta}'-\mathcal J_{1, \theta}'' \right\|_{L^p(\Omega)}^2\le c_{3,13}\varepsilon^{\frac{\alpha-1}{\alpha}\wedge \left[\theta\left(2\vartheta+\frac{\alpha+2H-2}{\alpha}\right)\right]}.
 \end{align*}
The proof is complete.
    \end{proof}

    \begin{lemma}\label{lem J14} For every $p> p_0$, there exists a finite positive constant $c_{3,20}$, which does not depend on   $\varepsilon \in (0,1)$, such that
          \begin{align}\label{eq J14 est}
           \sup_{t\in [0,T]}\sup_{x\in \mathbb R} \left\| \mathcal J_{1, \theta}''  - \widetilde{\mathcal J}_{1,\theta}'\right\|_{L^p(\Omega)}^2
           \le c_{3,20}\varepsilon^{\theta\left( \frac{\alpha+2H-2}{\alpha}+2\vartheta_0\right)}.
          \end{align}
          \end{lemma}
              \begin{proof}
            By the Cauchy-Schwarz inequality, \eqref{eq Holder},  Remark \ref{remark1}, the Lipschitz continuity of $\sigma$, BDG's inequality \eqref{eq BDG} and Lemma \ref{lem int p}, we have
            \begin{align*}\label{eq J1a 6}
        \left\|\mathcal J_{1, \theta}''-   \widetilde{\mathcal J}_{1,\theta} '\right\|_{L^{p}(\Omega)}^2
  \le &\, c_{3,21}  L_{\sigma}^2 \varepsilon^{2\vartheta_0\theta}  \left\|\int_{t(\varepsilon)}^t\int_{\mathbb R}G_{\alpha}([t-s,t+\varepsilon-s],x-y) W(ds,dy)\right\|_{L^{2p}(\Omega)}^2\\
  \le &\, 2c_{3,22}\varepsilon^{2\vartheta_0\theta} \int_{t(\varepsilon)}^t ds   \int_{\mathbb R}dy  \int_{\mathbb R}dz   \big|G_{\alpha}(t+\varepsilon-s,  x-y)- G_{\alpha}(t+\varepsilon-s,  x-z)  \Big|^2     |y-z|^{2H-2} \nonumber\\
 &\,+2c_{3,22}\varepsilon^{2\vartheta_0\theta} \int_{t(\varepsilon)}^t ds   \int_{\mathbb R}dy  \int_{\mathbb R}dz   \big|G_{\alpha}(t-s,  x-y)- G_{\alpha}(t-s,  x-z)  \Big|^2|y-z|^{2H-2}\nonumber \\
  \le &\, c_{3,23}\varepsilon^{2\vartheta_0\theta}  \int_{t(\varepsilon)}^t (t+\varepsilon-s)^{\frac{2H-2}{\alpha}}ds+c_{3,23} \varepsilon^{2\vartheta_0\theta} \int_{t(\varepsilon)}^t (t-s)^{\frac{2H-2}{\alpha}}ds
  \le  c_{3,24}\varepsilon^{2\vartheta_0\theta} \cdot \varepsilon^{\theta{\frac{\alpha+2H-2}{\alpha}}}.
 \end{align*}
The proof is complete.
          \end{proof}

\subsection{Proofs of the results in Section \ref{sec main}}
We prove those results by using the argument in \cite[Proposition 4.6]{KSXZ2013}, \cite[Theorem 1.2]{WX2024} and \cite[Theorem 1.2]{QianWAngXiao2026}.

\begin{proof}[Proof of Theorem \ref{thm main}]
  For any $\vartheta \in (0, \vartheta_0] \cap (0, \frac{1-2H}{2\alpha}]$, let
\begin{equation}\label{DEF_theta}
    \theta := \frac{1}{1+\vartheta}.
\end{equation}
For this particular choice of $\theta$,
\begin{equation*}
    \begin{split}
        \frac{\alpha+2H-2}{\alpha} &+ \left[ \frac{\alpha-2H+2}{\alpha}(1-\theta) \right] \wedge \frac{1-2H}{\alpha}
        = \frac{\alpha-1}{\alpha} \wedge \left( \theta \left( 2\vartheta + \frac{\alpha+2H-2}{\alpha} \right) \right)
        =: \mathcal{G}_{\vartheta}.
    \end{split}
\end{equation*}
  Note that $\mathcal{G}_{\vartheta} = \frac{\alpha+2H-2}{\alpha} + \left[ \frac{(\alpha-2H+2)\vartheta}{\alpha(1+\vartheta)} \wedge \frac{1-2H}{\alpha} \right] \in \left( \frac{\alpha+2H-2}{\alpha}, \frac{\alpha-1}{\alpha} \right]$.

   By Minkowski's inequality, Jensen's inequality and Lemmas \ref{lem J12} and  \ref{lem J14}, we obtain that for any $\eta_1 \in \left( \frac{\alpha+2H-2}{\alpha}, \mathcal{G}_{\vartheta} \right)$,
  \begin{align}\label{eq:J1_theta_est}
    \left\|  \mathcal J_{1,\theta}'- \widetilde{\mathcal J}_{1,\theta}' \right\|_{L^p(\Omega)} \leq c_{4,1} \, \varepsilon^{\frac{\eta_1}{2}},
  \end{align}
  for some finite positive constant $c_{4,1}$ independent of $\varepsilon \in (0,1)$, $t \in [0,T]$ and $x \in \mathbb{R}$.

  For any interval $Q \subset [0,T]$, denote $\Lambda(Q) := \int_{Q \times \mathbb{R}} \left[ G_{\alpha}(t+\varepsilon-s, x-y) - G_{\alpha}(t-s, x-y) \right] W(ds,dy)$. Then $\widetilde{\mathcal J}_{1,\theta}'  = \sigma(u(t, x)) \Lambda([t-\varepsilon^{\theta}, t])$ by \eqref{LAbJ1theta}. By \eqref{eq bound} and the Cauchy-Schwarz inequality, we obtain that
  \begin{align}\label{eq:Lambda_bound}
    \bigl\| \widetilde{\mathcal J}_{1,\theta}'  - \sigma(u(t, x)) \Lambda([0,t]) \bigr\|_{L^p(\Omega)}
    \leq c_{4,2} \bigl\| \Lambda([t-\varepsilon^{\theta}, t]) \bigr\|_{L^p(\Omega)}.
  \end{align}
  Since $\Lambda([0, t-\varepsilon^{\theta}])$ equals $\mathcal{J}_{1,\theta}$ when $\sigma = 1$, we may apply Lemma \ref{lem J11} to the linear equation \eqref{eq SFHE linear} in order to see that for any $\eta_1 \in \left( \frac{\alpha+2H-2}{\alpha}, \mathcal{G}_{\vartheta} \right)$,
  \begin{align}\label{eq:Lambda_est}
    \bigl\| \Lambda([0, t-\varepsilon^{\theta}]) \bigr\|_{L^p(\Omega)} \leq c_{4,3} \, \varepsilon^{\frac{\eta_1}{2}}.
  \end{align}

  Combining \eqref{eq:J1_theta_est}--\eqref{eq:Lambda_est} together, we obtain that for every $p \geq 1$ and $\eta_1 \in \left( \frac{\alpha+2H-2}{\alpha}, \mathcal{G}_{\vartheta} \right)$, there exists a positive constant $c_{4,4}$ independent of $\varepsilon \in (0,1)$, $t \in [0,T]$ and $x \in \mathbb{R}$, such that
  \begin{align}\label{Est_ADD1}
    \left\|  \widetilde{\mathcal J}_{1,\theta}' - \sigma\left( u(t(\varepsilon), x) \right) \int_{0}^{t} \int_{\mathbb R} \left[ G_{\alpha}(t+\varepsilon-s, x-y) - G_{\alpha}(t-s, x-y) \right] W(ds,dy) \right\| \leq c_{4,4} \, \varepsilon^{\frac{\eta_1}{2}}.
  \end{align}

  Together with \eqref{Est_ADD1}, Lemma \ref{lem J2}, Lemma \ref{lem J11} and Corollary \ref{cor tildeJ1 theta}, this implies that \eqref{eq main} holds for every $\eta \in \left( \frac{\alpha+2H-2}{\alpha}, \mathcal{G}_{\vartheta} \right)$. 

In addition, by Remark \ref{remark1},   the initial condition $u_0$ is H\"older continuous with order $ \beta_0$. If $ \beta_0>\frac{\alpha+2H-2}{\alpha}$, then
  by \cite[Proposition 2.6]{KS2023} and \eqref{eq scaling},  there exists some constant $c_{4,5}>0$   such that
   \begin{align}\label{eq J0_bound1}
|\mathcal{J}_0| = \left|\int_{\mathbb{R}} \bigl(G_\alpha(t+\varepsilon,x-y)-G_\alpha(t,x-y)\bigr) u_0(y) \, dy\right| \le c_{4,5} \, \e^{\frac{\beta_0}{2}}.
\end{align}
 Here, the positive  constant  $c_{4,5}$  does not depend on  $\e\in (0,1)$, $t\in [0,T]$   and $x\in \mathbb R$.
Putting   \eqref{eq main} and \eqref{eq J0_bound1} together, we obtain  \eqref{eq main 2}.  The proof of Theorem \ref{thm main}   is complete.
                   \end{proof}

Under Condition \ref{cond A}, for any fixed $t>0$ and $x, y\in \mathbb{R}$, by \cite[Lemma 2.2]{ChenZhang2016*}, we have
\begin{align*}
 \left|\int_{\mathbb{R}}G_{\alpha}(t,x-z)u_0(z)dz-\int_{\mathbb{R}}G_{\alpha}(t,y-z)u_0(z)dz\right|
  \leq c_{4,6}\frac{|x-y|}{t^{\frac{1}{\alpha}}}\left\|u_0(z)\right\|_{L^{\infty}(\mathbb{R})},
\end{align*}
  for some constant $c_{4,6}>0$.  Hence, by  \cite[Proposition 2.6]{KS2023}, there exists  a constant $c_{4,7}=c(T_1, T_2)>0$ such that for all $0<T_1\leq t \leq T_2$, $\e\in(0,1)$ and $x\in\mathbb{R}$,
   \begin{align}\label{eq J01}
  \left| \left[G_{\alpha}(t+\e,\cdot)-G_{\alpha}(t,\cdot)\right]\ast u_0(x)\right|\leq c_{4,7}\e^{{\frac{1}{2}}}.
   \end{align}
   Putting    \eqref{eq main} and  \eqref{eq J01} together,  there exists a positive constant $c_{4,8}>0$ such that for all   $\e>0$, $\eta\in \mathcal{I}$  and $T_2>T_1>0$,
        \begin{equation*}
        \begin{split}
      \sup_{t\in [T_1, T_2]}\sup_{x\in \mathbb R}  \Big\| (\mathcal{D}_{\e}u)(t,x)-\sigma(u(t, x)) (\mathcal{D}_{\e}v)(t,x)  \Big\|_{L^p(\Omega)} \le & c_{4,8}\,   \varepsilon^{\frac{\alpha+2H-2}{2\alpha}+\eta}.
      \end{split}
         \end{equation*}

\begin{proof}[Proof of Corollary \ref{cor  t LIL} ]

The results of   Corollary \ref{cor  t LIL} at $t>0$  follow  from Lemma  \ref{lem 51}, the decomposition \eqref{eq decom},     and  Khinchin's and Chung's laws of the iterated logarithm for the fBm (e.g., \cite[Section 7]{LS01*}, \cite{MR1995}, \cite{Tal94}).
 The results  at  the origin  follow  from  Lemma  \ref{lem 51},  Lemma \ref{corotLIL0}, and Theorem \ref{thm main}. The proof is complete.
\end{proof}

\section{Appendix}

The following technical lemma is taken from Liu and Mao \cite{LM2022}. Note that $G_{\alpha}(t,x)$ in this paper coincides with $G_{\alpha,\delta}(t,x)$ in \cite{LM2022}) in the case $\delta=0$.

\begin{lemma}\label{lem int p}\cite[Lemma 2]{LM2022} For any $\beta\in (0,1)$ and $s\in [0,T]$, there exists a positive constant $c_{5,1}$ such that
$$
\int_{\mathbb R} \left[\mathcal N_{\beta} G_{\alpha,\delta}(s,x) \right]^2dx \le c_{5,1} s^{-\frac{2\beta+1}{\alpha}}.
$$
\end{lemma}

By adapting the arguments of Lemma 5.1 and Proposition 5.1 in \cite{QianWAngXiao2026}, we obtain the following result.
\begin{lemma}\label{lem 51}
Assume Condition \ref{cond A} holds. Then for every fixed $t\in[0,T]$, $x\in\mathbb{R}$,
and every $\eta\in\mathcal{I}$, with probability one,
\begin{align*}
\sup_{0<\varepsilon<\delta} \big|(\mathcal{D}_{\e}u)(t,x)
-\big[G_{\alpha}(t+\varepsilon,\cdot)-G_{\alpha}(t,\cdot)\big]*u_0(x)
-\sigma(u(t, x)) (\mathcal{D}_{\e}v)(t,x)\big|
= O\big(\delta^{\frac{\alpha+2H-2}{2\alpha}+\eta}\big), \quad \text{as } \delta\downarrow 0.
\end{align*}
\end{lemma}

 \begin{lemma}\label{corotLIL0}
        Choose and fix   $x\in \mathbb R$.   Then,
        \begin{itemize}
        \item[(a).] (Khinchin's LIL)
        \begin{equation}\label{Eq:LIL0}
        \begin{split}
        \lim_{r \to 0}  \sup_{ 0<t<r}  \frac{|v(t, x)|}{t^{\widetilde{H}/2}\sqrt{2 \log\log(1/t)}}=  \widetilde\kappa, \ \ \ \text{a.s.},
        \end{split}
                 \end{equation}
where $\widetilde{H}=\frac{\alpha+2H-2}{\alpha}$ and $  \widetilde\kappa$ is given by  \eqref{eq con kappa1}.
            \item[(b).]  (Chung's LIL)
       \begin{equation}\label{Eq: CLIL0}
\begin{split}
\liminf_{\e \to 0}\sup_{0\le r\le \e}\frac{|v(r, x)|}{(\e / \log\log(1/\e))^{\widetilde{H}/2}}
= \left( \frac{ \Gamma(2\widetilde{H})}{ \Gamma(\widetilde{H})} \right)^{\frac12}
\lambda_{\widetilde{H}}^{\widetilde{H}/2}, \quad \text{a.s.}
\end{split}
\end{equation}
        where $\widetilde{H}=\frac{\alpha+2H-2}{\alpha}$ and $\lambda_H$ is the small ball constant of a  fBm with index $\frac{\widetilde{H}}{2}$
        (see, e.g., \cite[Theorem 6.9]{LS01*}).
      \end{itemize}
 \end{lemma}

\end{document}